\documentclass[12pt,english]{amsart}
\usepackage[T1]{fontenc}
\usepackage[latin9]{inputenc}
\usepackage{fancyhdr}
\usepackage{float}
\usepackage{amsthm}
\usepackage{graphicx}
\usepackage{amssymb}
\usepackage{esint}

\makeatletter
\numberwithin{equation}{section} 
\numberwithin{figure}{section} 
\theoremstyle{plain}
\theoremstyle{plain}
\theoremstyle{plain}
\newtheorem{thm}{Theorem}
  \theoremstyle{remark}
  \newtheorem*{rem*}{Remark}
\theoremstyle{plain}
\newtheorem{extthm}{Theorem}
            
  \theoremstyle{plain}
  \newtheorem{lem}[thm]{Lemma}
  \theoremstyle{plain}
  \newtheorem{prop}[thm]{Proposition}
  \theoremstyle{plain}
  \newtheorem{cor}[thm]{Corollary}

\usepackage{tocvsec2}
\maxtocdepth{subsection}

\makeatother

\usepackage{babel}

\begin{document}
\global\long\def\A{{\mathcal{A}}}

\global\long\def\B{{\mathcal{B}}}

\global\long\def\C{{\mathcal{C}}}

\global\long\def\M{{\mathcal{M}}}

\global\long\def\N{{\mathcal{N}}}

\global\long\def\O{{\mathcal{O}}}

\global\long\def\bbc{{\mathbb{C}}}

\global\long\def\bbp{{\mathbb{P}}}

\global\long\def\bbd{{\mathbb{D}}}

\global\long\def\bbt{{\mathbb{T}}}

\global\long\def\bbr{{\mathbb{R}}}

\global\long\def\bbs{{\mathbb{S}}}

\global\long\def\bbn{{\mathbb{N}}}

\global\long\def\dprime{{\prime\prime}}

\global\long\def\vol#1{\mathrm{vol}(#1)}

\global\long\def\volcn#1{\mathrm{vol_{\bbc^{N}}}(#1)}

\global\long\def\volrn#1#2{\mathrm{vol_{\bbr^{#1}}}(#2)}


\global\long\def\disk#1{#1 \bbd}

\global\long\def\cir#1{#1 \bbt}


\global\long\def\pr#1{\bbp\left( #1 \right)}

\global\long\def\Ex{\mathbb{E}}

\global\long\def\Var#1{\mathrm{Var}(#1)}

\global\long\def\Cov#1{\mathrm{Cov}(#1)}


\global\long\def\Li{{\mathrm{Li}}}

\global\long\def\erf{\mathrm{erf}}

\title{The Hole Probability for\\
Gaussian Entire Functions}

\author{Alon Nishry}

\address{School of Mathematical Sciences, Tel Aviv University, Tel Aviv 69978,
Israel }

\email{alonnish@post.tau.ac.il}
\begin{abstract}
Consider the random entire function\[
f(z)=\sum_{n=0}^{\infty}\phi_{n}a_{n}z^{n},\]
 where the $\phi_{n}$ are independent standard complex Gaussian coefficients,
and the $a_{n}$ are positive constants, which satisfy\[
\lim_{n\to\infty}\frac{\log a_{n}}{n}=-\infty.\]

We study the probability $P_{H}(r)$ that $f$ has no zeroes in the
disk $\left\{ |z|<r\right\} $ (hole probability). Assuming that the
sequence $a_{n}$ is logarithmically concave, we prove that\[
\log P_{H}(r)=-S(r)+o\left(S(r)\right),\]
where\[
S(r)=2\cdot\sum_{n\,:\, a_{n}r^{n}\ge1}\log\left(a_{n}r^{n}\right),\]
and $r$ tends to $\infty$ outside a (deterministic) exceptional
set of finite logarithmic measure.
\end{abstract}

\thanks{Research supported by the Israel Science Foundation of the Israel
Academy of Sciences and Humanities, grant 171/07.}

\maketitle

\section{Introduction}

Consider the random entire function\begin{equation}
f(z)=\sum_{n=0}^{\infty}\phi_{n}a_{n}z^{n},\label{eq:f_def}\end{equation}
where the $\phi_{n}$'s are independent standard complex Gaussian
coefficients and the $a_{n}$'s are positive constants, such that\[
\lim_{n\to\infty}\frac{\log a_{n}}{n}=-\infty.\]
The latter condition guarantees that almost surely the series on the
right-hand side of \eqref{eq:f_def} has infinite radius of convergence.
The probability $P_{H}(r)$ of the event that $f$ has no zeros in
the disk $\left\{ |z|\le r\right\} $ is called the hole probability.
We are interested in the decay rate of the hole probability as $r$
grows to infinity.

This question was studied by Sodin and Tsirelson \cite{ST3} for a
special choice of the coefficients $a_{n}=\frac{1}{\sqrt{n!}}$ (see
also the earlier paper \cite{Sod}, for an approach to the problem
in a more general setting). Their work was continued in \cite{Nis},
where we gave more precise estimates for the hole probability. Since
the technique in \cite{Nis} was mostly independent of the special
choice of the coefficients $a_{n}$, it led naturally to the generalizations
in this paper. Here, we combine ideas introduced in \cite{ST3} and
\cite{Nis} with the classical Wiman-Valiron theory of growth of power
series.

To state the main result, we need to introduce two functions which
depend on the coefficients $a_{n}$. The first $\N_{1}(r)$ is the
set that contains the {}``significant'' coefficients of $f(z)$,
for the given value of $r$\[
\N_{1}(r)=\left\{ n\,:\,\log\left(a_{n}r^{n}\right)\ge0\right\} ,\]
we also write\[
N_{1}(r)=\#\N_{1}(r).\]
The second function is\[
S(r)=\log\left(\prod_{n\in\N_{1}(r)}\left(a_{n}r^{n}\right)^{2}\right)=2\cdot\sum_{n\in\N_{1}(r)}\log\left(a_{n}r^{n}\right).\]
For the sake of simplicity of the presentation, we will assume that
the coefficients $a_{n}$ are the restriction of a (real, positive)
function $a(t)\in C^{2}\left(\left(0,\infty\right)\right)$ to the
set of natural integers (it is clear though, that we can interpolate
any such sequence with a smooth function). In order for $f(z)$ to
be an entire function we require the following\[
\lim_{t\to\infty}\frac{\log a(t)}{t}=-\infty.\]
In addition, we require that $a(t)$ is a log-concave function. We
also use the notation\[
p_{H}(r)=-\log P_{H}(r)=\log^{-}P_{H}(r).\]
A measurable set $E\subset\left[1,\infty\right)$ has a finite logarithmic
measure if\[
\intop_{E}\frac{1}{t}\, dt<\infty.\]
\newpage{}The following is our main result
\begin{thm}
\label{thm:p_H_raw_asymp}Suppose that $a(t)$ is a log-concave function.
For $r\to\infty$ not belonging to a (deterministic) set of finite
logarithmic measure\[
p_{H}(r)=S(r)+o\left(S(r)\right).\]

\end{thm}
We do not know whether the log-concavity condition is essential for
this result. If, in addition, we have some lower bound condition on
the function $a(t)$ (i.e. the function $f$ does not grow too slowly),
we can say more, for example we can prove
\begin{thm}
\label{thm:p_H_exact_asymp}Let $\alpha\ge1$. If $a(t)$ is log-concave
and $a(t)\ge\exp\left(-t\log^{\alpha}t\right)$, then there exist
positive absolute constants $c_{1}$ and $c_{2}$, such that for any
$\epsilon>0$ and for $r$ not belonging to a set of finite logarithmic
measure\[
S(r)-c_{1}\left(S(r)\right)^{0.9+\epsilon}\le p_{H}(r)\le S(r)+c_{2}\left(S(r)\right)^{0.5+\epsilon}.\]
\end{thm}
\begin{rem*}
If the coefficient $a_{n}$ are given in explicit form, then it is
possible to prove results that are true for every value of $r$ that
is large enough, using direct computations instead of Wiman-Valiron
theory. As an example, one can take Mittag-Leffler coefficients ($\alpha>0$)
\[
a_{n}=\frac{1}{\Gamma\left(\alpha n+1\right)},\]
in that case\[
p_{H}(r)=\frac{1}{2\alpha}r^{2/\alpha}+\O_{\alpha}\left(r^{9/5\alpha}\right).\]
We do not reproduce these calculations here, since they are very similar
to the general ones. See the paper \cite{Nis} for the case $a_{n}=\frac{1}{\sqrt{n!}}$.
\end{rem*}
\textit{Acknowledgment:} This work is based on part of my master's
thesis, which was written in Tel Aviv University. I would like to
thank my advisor Mikhail Sodin for his guidance and encouragement
throughout my studies. I also thank Manjunath Krishnapur and the referee
for numerous remarks to the preliminary version of this paper, that
significantly improved the presentation.

\section{Preliminaries\label{sec:Main-Thm-Prelim}}

\subsection{Notation}

We denote by $\disk r$ the disk $\left\{ z\,:\,|z|<r\right\} $ and
by $\cir r$ its boundary $\left\{ z\,:\,|z|=r\right\} $, with $r\ge1$.
The letters $c$ and $C$ denote positive absolute constants (which
can change across lines). We use the notation $\log_{n}^{m}$ as a
shortcut for the $n$ times iterated logarithm, taken to the $m$-th
power (i.e. $\log_{2}^{2}(x)\equiv\left(\log\log(x)\right)^{2}$ and
$\log_{1}^{m}x$ is written as $\log^{m}x$).

In order to simplify some of the expressions in the paper, we will
assume from now on that \[
a_{0}=a(0)=1.\]

\subsection{Results from Wiman-Valiron theory}

Let $g(z)$ be a transcendental entire function given in the form\[
g(z)=\sum_{n=0}^{\infty}a_{n}z^{n}.\]
We recall some of the results of Wiman-Valiron theory, taken from
\cite{Ha1} and \cite[Section 6.5]{Ha2}. Let $r\ge0$, we denote
by $M(r)$ the maximum of $g(z)$ inside $\disk r$, by $\mu(r)$
the maximal term of $g(z)$\[
\mu(r)=\max_{n}|a_{n}|r^{n}\]
and by $\nu(r)$ the (maximal) index of the maximal term $\mu(r)$
(Hayman's survey uses the notation $N(r)$ for this function). For
every transcendental entire function we have $\mu(r)\to\infty$ and
$\nu(r)\to\infty$ as $r\to\infty$. We note that the maximal index
and the maximal term are related to each other by a simple equation
(see \cite[p. 318]{Ha1})\begin{equation}
\log\mu(r)=\log\mu(1)+\intop_{1}^{r}\frac{\nu(t)}{t}\, dt\qquad r\ge1.\label{eq:integral_rel_mu_nu}\end{equation}
We will give the following simple example: Take $a_{n}=\frac{1}{n!}$,
and so $g(z)=e^{z}$. The maximal term can be estimated using Stirling's
approximation:\begin{eqnarray*}
\frac{r^{n}}{n!} & = & \frac{r^{n}}{\sqrt{2\pi n}\left(\frac{n}{e}\right)^{n}}\left(1+o(1)\right)\Rightarrow\\
\nu(r) & = & r+O(1),\\
\mu(r) & = & \frac{e^{r}}{\sqrt{2\pi r}}\left(1+o(1)\right),\end{eqnarray*}
so there is an asymptotic agreement with \eqref{eq:integral_rel_mu_nu}.
Notice that we also have\[
\log M(r)=\log\mu(r)+o(\log\mu(r)).\]
Most of the statements in Wiman-Valiron theory include a positive
decreasing function $b(m)$, which satisfies\[
\intop_{1}^{\infty}b(m)\, dm<\infty.\]
Here, we always use the function\[
b(m)=\frac{1}{m\log^{2}m}.\]
The following theorem (\cite[p. 322]{Ha1}) bounds from above the
values of the terms away from the maximal term for values of $r$,
outside a set of \textit{finite logarithmic measure} (\textit{FILM}). 
\begin{extthm}
\label{thm:WV-upp-bnd-for-terms}Set $n=k+\nu(r)$. If $r$ is outside
a set of FILM then\begin{equation}
\frac{|a_{n}|r^{n}}{\mu(r)}\le\exp\left(-ck^{2}\cdot b\left(|k|+\nu(r)\right)\right),\label{eq:upp-bnd-for-terms}\end{equation}
with $c$ a positive absolute constant (notice that the exceptional
set depends only on the $a_{n}$'s).
\end{extthm}
The most famous result in this theory \cite[p. 333]{Ha1}, gives an
estimate for $M(r)$ in terms of $\mu(r)$ and $\nu(r)$:
\begin{extthm}
For all sufficiently large values of $r$, outside a set of FILM\begin{equation}
M(r)<\mu(r)\log^{1/2}\mu(r)\log_{2}^{2}\mu(r).\label{eq:WV_main_thm}\end{equation}

\end{extthm}
We will use the theorem above in the most basic way, claiming that
$\log M(r)=\log\mu(r)+o\left(\log\mu(r)\right)$ for large values
of $r$ outside a set of FILM. We also borrow the following result
from \cite[p. 360]{Ha2}:
\begin{extthm}
Outside a set of FILM\begin{equation}
\nu(r)<\log\mu(r)\log_{2}^{2}\mu(r).\label{eq:upp-bnd-nu}\end{equation}

\end{extthm}
From now on we will call $r$ \textit{normal} if it satisfies both
\eqref{eq:upp-bnd-for-terms} and \eqref{eq:upp-bnd-nu}, this again
holds outside a set of FILM.

\subsection{The function $N_{x}(r)$}

We use the following notation:\[
\N_{x}(r)=\left\{ n\::\:\log\left(a_{n}r^{n}\right)\ge\left(1-x\right)\log\mu(r)\right\} ,\qquad x\ge0,\]
and\[
N_{x}(r)=\#\N_{x}(r).\]
Also\[
\N_{m,m+1}(r)=\N_{m+1}(r)\backslash\N_{m}(r)\]
and $N_{m,m+1}(r)$ is size of $\N_{m,m+1}(r)$. Note that if $n\in\N_{m,m+1}(r)$
than\begin{equation}
a_{n}r^{n}\le\mu^{1-m}(r).\label{eq:a_n_r^n_upp_bnd}\end{equation}
We also partition the ``tail'' indexes, $\left(\N_{1}(r)\right)^{c}$,
into a union of sets $\bigcup_{m=1}^{\infty}\N_{m,m+1}(r)$.

We will use the fact that $a(t)$ is a log-concave function to derive
some properties of $N_{x}(r)$ and $N_{m,m+1}(r)$. We use the function
\[
h(t)=\log a(t)+t\log r,\]
 note that it is concave since $a(t)$ is log-concave. Now denote
by $N_{1}^{\prime}(r)$ the largest root of the equation $h\left(t\right)=0$,
we see that $N_{1}(r)=\left[N_{1}^{\prime}(r)\right]+1$, and in particular
$N_{1}^{\prime}(r)<N_{1}(r)\le N_{1}^{\prime}(r)+1$. If we draw the
line from the point $\left(\nu(r),\log\mu(r)\right)$ to the point
$\left(N_{1}^{\prime}(r),0\right)$, then it satisfies the following
equation\begin{equation}
y(t)=\frac{\log\mu(r)}{N_{1}^{\prime}(r)-\nu(r)}\cdot\left(N_{1}^{\prime}(r)-t\right).\label{eq:Line-equ-mu-N_1}\end{equation}
It will be useful to keep in mind the following picture

\begin{figure}[H]
\includegraphics[scale=0.35]{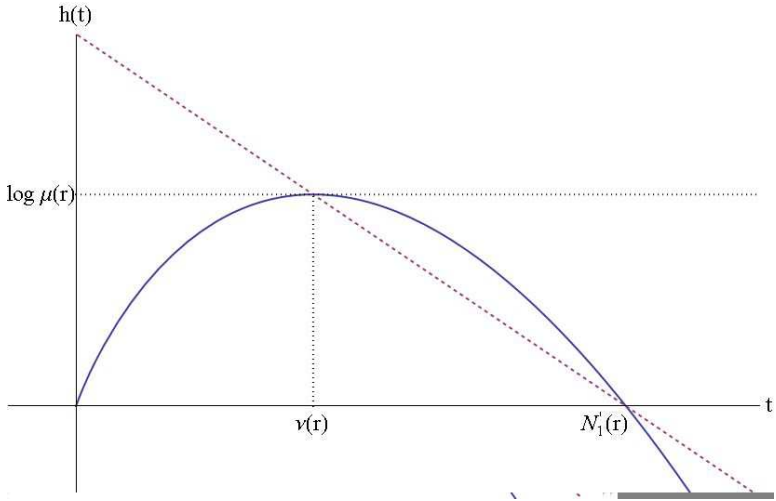}

\caption{The graph of $h(t)$. The dashed line is $y(t)$.\label{fig:func_h_t}}

\end{figure}
The following lemma gives an estimate for the tail of $h(t)$.
\begin{lem}
For $t\ge N_{1}^{\prime}(r)$, we have.\label{lem:Nx-tail-properties}\[
h(t)\le\frac{\log\mu(r)}{N_{1}^{\prime}(r)-\nu(r)}\cdot\left(N_{1}^{\prime}(r)-t\right)\]
and for $x\ge1$ we have\[
N_{x}(r)\le xN_{1}(r).\]
\end{lem}
\begin{proof}
Looking at the picture above, we see that for $t\ge N_{1}^{\prime}(r)$
the function $h(t)$ lies under the line given by \eqref{eq:Line-equ-mu-N_1},
and we get the first result. The second part follows from the log-concavity
of $a(t)$, since\begin{eqnarray*}
h\left(xN_{1}(r)\right) & \le & y\left(xN_{1}(r)\right)=\frac{\log\mu(r)}{N_{1}^{\prime}(r)-\nu(r)}\cdot\left(N_{1}^{\prime}(r)-xN_{1}(r)\right)\\
 & = & \frac{\log\mu(r)}{N_{1}^{\prime}(r)-\nu(r)}\cdot\left(N_{1}^{\prime}(r)-\nu(r)+\nu(r)-xN_{1}(r)\right)\\
 & = & \log\mu(r)-\log\mu(r)\cdot\left(\frac{xN_{1}(r)-\nu(r)}{N_{1}^{\prime}(r)-\nu(r)}\right)\\
 & < & \left(1-x\right)\log\mu(r).\end{eqnarray*}
The last inequality is true since $N_{1}(r)>N_{1}^{\prime}(r)$.
\end{proof}
It follows immediately from the previous lemma that for $m\ge1$ we
have \begin{equation}
N_{m,m+1}(r)\le mN_{1}(r).\label{eq:N_m,m+1_upper_bnd}\end{equation}
 We will now use Wiman-Valiron theory to find an upper bound for $N_{1}(r)$
in terms of $\log\mu(r)$.
\begin{lem}
\label{cor:N_1-upper_bnd}For large normal values of $r$, we have\begin{equation}
N_{1}(r)<C\log\mu(r)\log_{2}^{2}\mu(r),\label{eq:N_x-upp_bnd}\end{equation}
with $C>1$ some positive absolute constant.\end{lem}
\begin{proof}
We use Theorem \ref{thm:WV-upp-bnd-for-terms} with $n=k+\nu(r)$
and $k>0$ and get\[
\log a_{n}r^{n}\le\log\mu(r)-ck^{2}\left(n\log^{2}n\right)^{-1}.\]
We now put $n=\left\lfloor C\log\mu(r)\log_{2}^{2}\mu(r)\right\rfloor $,
with some $C>1$, to be selected later. Notice that using \eqref{eq:upp-bnd-nu}
we have\[
k=n-\nu(r)\ge(C-1)\log\mu(r)\log_{2}^{2}\mu(r).\]
We also note that for $r$ large enough\[
\log^{2}n\le2\cdot\log_{2}^{2}\mu(r).\]
We now have the following inequality\[
\log a_{n}r^{n}\le\log\mu(r)-\frac{(C-1)^{2}}{C\cdot2}\cdot\log\mu(r)\le0,\]
with a suitable choice of the constant $C$.
\end{proof}
We will also use the following lower bound for $N_{1}(r)$,
\begin{lem}
\label{lem:N_1_low_bound}We have\[
N_{1}(r)\ge\nu(r)\ge\frac{\log\mu(r)-\log\mu(1)}{\log r}.\]
\end{lem}
\begin{proof}
The left inequality follows from the fact that $h(t)$ is concave.
The right inequality follows from \eqref{eq:integral_rel_mu_nu}.
We remark that as a conclusion we see that $N_{1}(r)\to\infty$ as
$r\to\infty$.
\end{proof}

\subsection{Properties of $S(r)$}

Looking again at Figure \ref{fig:func_h_t}, we clearly have\[
S(r)=2\cdot\sum_{n\in\N_{1}(r)}\log\left(a_{n}r^{n}\right)\ge N_{1}^{\prime}(r)\cdot\max_{n\in\N_{1}(r)}\log\left(a_{n}r^{n}\right)=N_{1}^{\prime}(r)\cdot\log\mu(r),\]
or\begin{equation}
S(r)\ge\left(N_{1}(r)-1\right)\cdot\log\mu(r).\label{eq:S(r):low:bnd}\end{equation}
Notice that similarly we also have\begin{equation}
S(r)\le2N_{1}(r)\log\mu(r)\le C\log^{2}\mu(r)\log_{2}^{2}\mu(r).\label{eq:S(r):upp:bnd}\end{equation}

\subsection{Gaussian Distributions}

We frequently use the fact that if $a$ has a $N_{\bbc}(0,1)$ distribution,
we have \begin{equation}
\pr{|a|\ge\lambda}=\exp(-\lambda^{2}),\label{eq:Gaus_prob_large}\end{equation}
and for $\lambda\le1,$

\begin{equation}
\pr{|a|\le\lambda}\in\left[\frac{\lambda^{2}}{2},\lambda^{2}\right].\label{eq:Gaus_prob_small}\end{equation}

\section{Upper Bound for $p_{H}(r)$\label{sec:Main-Thm-Upper-Bound}}

In this section we prove the following
\begin{prop}
\label{thm:p_H_upp_bnd}For normal values of $r$, we have \[
p_{H}(r)\le S(r)+C\cdot N_{1}(r)\log N_{1}(r),\]
with $C$ some positive absolute constant.\end{prop}
\begin{rem*}
We note that $r$ is assumed to be large. Later we will analyze the
error term.
\end{rem*}
The simplest case where $f(z)$ has no zeros inside $\disk r$ is
when the constant term dominates all the others. We therefore study
the event $\Omega_{r}$, which is the intersection of the events $\mbox{{\rm \mbox{(i)}}}$,$\mbox{(ii)}$
and $\mbox{(iii)}$, where\[
\begin{array}{ll}
\mbox{{\rm \mbox{(i)}}}: & |\phi_{0}|\ge\sqrt{N_{1}(r)}+3,\\
\mbox{{\rm \mbox{(ii)}}}: & {\displaystyle \bigcap_{n\in\N_{1}(r)}}\mbox{{\rm \mbox{(ii)}}}_{n},\\
\mbox{{\rm \mbox{(iii)}}}: & {\displaystyle \bigcap_{m\in\left\{ 1,2,\ldots\right\} }}\mbox{{\rm \mbox{(iii)}}}_{m,m+1},\\
\mbox{{\rm \mbox{(iii)}}}_{m,m+1}: & {\displaystyle \bigcap_{n\in\N_{m,m+1}(r)}}\mbox{{\rm \mbox{(iii)}}}_{m,m+1,n},\end{array}\]
and\[
\begin{array}{ll}
\mbox{{\rm \mbox{(ii)}}}_{n}: & |\phi_{n}|\le\frac{(a_{n}r^{n})^{-1}}{\left(N_{1}(r)\right)^{1/2}},\\
\mbox{{\rm \mbox{(iii)}}}_{m,m+1,n}: & |\phi_{n}|\le\frac{\mu(r)^{m-1}}{N_{m,m+1}(r)\cdot m^{2}}.\end{array}\]

\begin{lem}
\label{lem:Low-Bnd-Est-For-Func}If $\Omega_{r}$ holds, then $f$
has no zeros inside $\disk r$.\end{lem}
\begin{proof}
To see that $f(z)$ has no zeros inside $\disk r$ we note that \begin{equation}
|f(z)|\ge|\phi_{0}|-\sum_{n=1}^{\infty}|\phi_{n}|a_{n}r^{n}.\label{eq:f_low_bnd}\end{equation}
First we estimate the sum over the terms in $\N_{1}(r)\backslash\left\{ 0\right\} $
\[
\sum_{n\in\N_{1}(r)\backslash\left\{ 0\right\} }|\phi_{n}|a_{n}r^{n}\le\sum_{n\in\N_{1}(r)}N_{1}(r)^{-\frac{1}{2}}=N_{1}(r)^{\frac{1}{2}}.\]
Now the tail is bounded by (using \eqref{eq:a_n_r^n_upp_bnd})\[
\sum_{n\in\N_{1}^{c}(r)}|\phi_{n}|a_{n}r^{n}=\sum_{m=1}^{\infty}\left[\sum_{n\in\N_{m,m+1}(r)}|\phi_{n}|a_{n}r^{n}\right]\le\sum_{m=1}^{\infty}\left[\sum_{n\in\N_{m,m+1}(r)}\left(N_{m,m+1}(r)\right)^{-1}m^{-2}\right]\]
and we have \begin{eqnarray}
\sum_{n\in\N_{1}^{c}(r)}|\phi_{n}|a_{n}r^{n} & \le & \sum_{m=1}^{\infty}\frac{1}{m^{2}}<2.\label{eq:sum_tail_upper_bnd}\end{eqnarray}
From \eqref{eq:f_low_bnd}\[
|f(z)|>\sqrt{N_{1}(r)}+3-N_{1}(r)^{1/2}-2=1,\]
we have that $f(z)\ne0$ inside $\disk r$. \end{proof}
\begin{lem}
\label{lem:Prob-of-Low-Bnd-Event}The probability of the event $\Omega_{r}$
is bounded from below by\textup{\[
\log\pr{\Omega_{r}}\ge-S(r)-C\cdot N_{1}(r)\log N_{1}(r),\]
}for normal values of $r$ which are large enough.\end{lem}
\begin{proof}
In the calculations we use the estimates \eqref{eq:Gaus_prob_large}
and \eqref{eq:Gaus_prob_small}. First we have

\[
\pr{{\rm \mbox{(i)}}}\ge\exp(-N_{1}(r)-2C\sqrt{N_{1}(r)}).\]
For the second part since $a_{n}r^{n}\ge1$,

\[
\pr{\mbox{{\rm \mbox{(ii)}}}_{n}}\ge\frac{(a_{n}r^{n})^{-2}}{2N_{1}(r)}\]
and so \begin{eqnarray*}
\pr{{\rm \mbox{(ii)}}} & \ge & \prod_{n\in\N_{1}(r)\backslash\left\{ 0\right\} }\frac{(a_{n}r^{n})^{-2}}{2N_{1}(r)}\\
 & \ge & \left(\prod_{n\in\N_{1}(r)}\frac{1}{(a_{n}r^{n})^{2}}\right)\exp\left(-N_{1}(r)\log N_{1}(r)+C\cdot N_{1}(r)\right)\\
 & \ge & \exp\left(-S(r)-C\cdot N_{1}(r)\log N_{1}(r)\right).\end{eqnarray*}
We handle the terms of ${\rm \mbox{(iii)}}$ separately for the first
term and the rest. For $m=1$, we have,\[
|\phi_{n}|\le\frac{1}{N_{1,2}(r)}\]
and so (using \eqref{eq:N_m,m+1_upper_bnd})\begin{eqnarray*}
\pr{\mbox{{\rm \mbox{(iii)}}}_{1,2}} & \ge & \left(\frac{1}{2\cdot\left(N_{1,2}(r)\right)^{2}}\right)^{N_{1,2}(r)}\\
 & \ge & \exp\left(-C\cdot N_{1,2}(r)\log N_{1,2}(r)\right)\\
 & \ge & \exp\left(-C\cdot N_{1}(r)\log N_{1}(r)\right).\end{eqnarray*}
For a fixed $m\ge2$ and $n\in\N_{m,m+1}$, we have\[
\pr{{\rm \mbox{(iii)}}_{m,m+1,n}}=1-\exp\left(-\frac{\mu(r)^{2(m-1)}}{\left(N_{m,m+1}(r)\right)^{2}\cdot m^{4}}\right).\]
We use the following inequality (for some positive sequence $\left\{ A_{n}\right\} $)\[
\pr{\forall n\,:\,|\phi_{n}|\le A_{n}}=1-\pr{\exists n\,:\,|\phi_{n}|>A_{n}}\ge1-\sum\pr{|\phi_{n}|>A_{n}}.\]
Using this inequality, we have\begin{equation}
\pr{{\rm \mbox{(iii)}}_{m\ge2}}\ge1-\sum_{m=2}^{\infty}N_{m,m+1}(r)\cdot\exp\left(-\frac{\mu(r)^{2(m-1)}}{\left(N_{m,m+1}(r)\right)^{2}\cdot m^{4}}\right)=1-\Sigma_{1}.\label{eq:far_tail_bound}\end{equation}
Taking $r$ which is normal and large enough we now have (using \eqref{eq:N_m,m+1_upper_bnd}
and Lemma \ref{cor:N_1-upper_bnd})\[
\Sigma_{1}\le C\cdot N_{1}(r)\cdot\sum_{m=1}^{\infty}m\cdot\exp\left(-\frac{\mu(r)^{2m-1}}{m^{6}}\right),\]
the first term in the sum is clearly the dominant one, and so\begin{eqnarray}
\pr{{\rm \mbox{(iii)}}} & \ge & 1-\exp\left(-c_{1}\mu(r)+C_{2}\log N_{1}(r)\right)\nonumber \\
 & \ge & 1-\exp\left(-c\mu(r)\right).\label{eq:tail_prob_lower_bnd}\end{eqnarray}
For our purposes here it is sufficient that $\pr{\mbox{(iii)}}$ is
larger than some absolute constant.

Since the $\phi_{n}$ are independent, we find that\[
\pr{\Omega_{r}}=\pr{\mbox{(i)}}\pr{\mbox{(ii)}}\pr{\mbox{(iii)}}\ge\exp\left(-S(r)-C\cdot N_{1}(r)\log N_{1}(r)\right)\]
and the lemma is proved.
\end{proof}
Proposition \ref{thm:p_H_upp_bnd} now follows from the previous lemmas.

\section{Lower Bound for $p_{H}(r)$\label{sec:Main-Thm-Lower-Bound}}

In this section we prove the following theorem
\begin{prop}
\label{thm:p_H_low_bnd}(Lower bound) Let $\delta\in\left(0,1\right)$.
For normal values of $r$, and for values of $\delta$ which satify
$\delta^{-4}=o\left(N_{1}(r)\right)$, we have\[
p_{H}(r)\ge S(\left(1-\delta\right)r)-C_{1}\cdot N_{1}(r)\log\log\mu(r)-C_{2}\delta^{-4}\log\mu(r)\]
where $C_{1},\, C_{2}$ are positive absolute constants.\end{prop}
\begin{rem*}
In principle it is possible to select $\delta^{-4}=cN_{1}(r)$, for
some constant $c>0$, but, we notice that in this case the error term
will be of the same order of magnitude as the main term (using \ref{eq:S(r):upp:bnd}).

Recall that for the lower bound we study the event in which $f$ doesn't
vanish in $\disk r$ (for large values of $r$). We define the deterministic
counterpart of $f(z)$, \[
\psi(z)=\sum_{n=0}^{\infty}a_{n}z^{n}\]
and write $M(r)=\max_{|z|\le r}|\psi(z)|=\sum_{n=0}^{\infty}a_{n}r^{n}$,
we also set $\M(r)={\displaystyle \max_{|z|\le r}|f(z)|}$. We start
by studying the deviations of $\log\M(r)$ from $\log M(r)$. Then
we consider large deviations of the expression \[
{\displaystyle \intop_{\cir r}\log|f(z)|\, dm},\]
where $m$ is the normalized angular measure on $\cir r$. Finally,
we use the fact that if $f(z)\ne0$ in $\disk r$ then $\log|f(z)|$
is a harmonic function inside $\disk r$, to get the result. 
\end{rem*}

\subsection{Large deviations for $\log\M(r)$}

We expect that $\log\,\M(r)$ will be very close to $\log M(r)$ with
high probability, but we don't need this accuracy for the lower bound.
In the next lemma we prove that the probability that $\log\,\M(r)$
will be large relatively to $\log M(r)$ is very small.
\begin{lem}
\textup{\label{lem:Dev-M-Upp-Bnd}Let $0<\sigma\le\frac{1}{2}$. Then
\[
\log\pr{\frac{\log\,\M(r)}{\log M(r)}\ge1+\sigma}\le-c\mu^{2\sigma}(r),\]
for normal values of $r$ which are large enough.}\end{lem}
\begin{proof}
We will construct an event with probability close to one, for which
$\log\M(r)$ is bounded by $(1+\sigma)\log M(r)$. Denote by $\Omega_{r}$
the event which is the intersection between the events ${\rm \mbox{(i)}},{\rm \mbox{(ii)}}$,
where\[
\begin{array}{ll}
\mbox{{\rm \mbox{(i)}}}: & {\displaystyle \bigcap_{n\in\N_{1}(r)}}\mbox{{\rm \mbox{(i)}}}_{n},\\
\mbox{{\rm \mbox{(ii)}}}: & {\displaystyle \bigcap_{m\in\left\{ 1,2,\ldots\right\} }}\mbox{{\rm \mbox{(ii)}}}_{m,m+1},\\
\mbox{{\rm \mbox{(ii)}}}_{m,m+1}: & {\displaystyle \bigcap_{n\in\N_{m,m+1}(r)}}\mbox{{\rm \mbox{(ii)}}}_{m,m+1,n},\end{array}\]
where

\[
\begin{array}{ll}
{\rm \mbox{(i)}}_{n} & |\phi_{n}|\le\mu^{\sigma}(r),\\
\mbox{{\rm {\rm \mbox{(ii)}}}}_{m,m+1,n} & |\phi_{n}|\le\frac{\mu(r)^{m-1}}{N_{m,m+1}(r)\cdot m^{2}}.\end{array}\]
We notice that\[
\pr{\mbox{{\rm \mbox{(i)}}}_{n}^{c}}=\exp\left(-\mu^{2\sigma}(r)\right).\]
In the proof of Lemma \ref{lem:Prob-of-Low-Bnd-Event}, we showed
(see \eqref{eq:tail_prob_lower_bnd}) \[
\pr{\mbox{{\rm {\rm \mbox{(ii)}}}}}\ge1-\exp\left(-c\mu(r)\right).\]
Therefore the probability that $\Omega_{r}$ does not occur is bounded
by \[
\pr{\Omega_{r}^{c}}\le\exp\left(-c\mu(r)\right)+N_{2}(r)\cdot\pr{\mbox{{\rm \mbox{(i)}}}_{n}^{c}}.\]
Using Lemma \ref{lem:Nx-tail-properties} and \ref{cor:N_1-upper_bnd}
we have, for $r$ large enough\[
\pr{\Omega_{r}^{c}}\le\exp\left(-c\mu^{2\sigma}(r)\right).\]
It is now sufficient to prove that for functions satisfying the above
inequalities, we have the aforementioned upper bound. Indeed\[
|f(z)|\le\sum_{n\in\N_{2}(r)}|\phi_{n}|a_{n}r^{n}+\sum_{n\in\N_{2}^{c}(r)}|\phi_{n}|a_{n}r^{n},\]
in \eqref{eq:sum_tail_upper_bnd} we already found that some absolute
constant is an upper bound for the second summand. The first summand
is bounded by\[
\sum_{n\in\N_{2}(r)}|\phi_{n}|a_{n}r^{n}\le\mu^{\sigma}(r)\cdot\sum_{n=0}^{\infty}a_{n}r^{n}=\mu^{\sigma}(r)\cdot M(r),\]
and so, for $r$ large enough (since $M(r)\ge\mu(r)+1$) \[
|f(z)|\le\mu^{\sigma}(r)\cdot M(r)+C\le M^{1+\sigma}(r).\]

\end{proof}
In the next lemma we prove that probability that $\log\,\M(r)$ will
be small is also very small.
\begin{lem}
\label{lem:Dev-M-Very-Small}We have\[
\log\pr{\log\,\M(r)\le0}\le-S(r).\]
\end{lem}
\begin{proof}
Suppose that $\log|f(z)|\le0$ in $\disk r$, using Cauchy's estimate
for the coefficients of $f(z)$ we can get an estimate to the probability
of this event. We have\[
|\phi_{n}|a_{n}r^{n}\le\M(r)\le1,\]
therefore for $n\in\N_{1}(r)$ we have\[
\pr{|\phi_{n}|\le\left(a_{n}r^{n}\right)^{-1}}\le\left(a_{n}r^{n}\right)^{-2},\]
and so\[
\pr{\log\M(r)\le0}\le\prod_{n\in\N_{1}(r)}\left(a_{n}r^{n}\right)^{-2}=\exp\left(-S(r)\right).\]

\end{proof}

\subsection{Discretization of the logarithmic integral}

In this section $N\ge1$ and $\delta\in(0,1)$ are fixed, $\kappa=1-\delta$
and the points $\left\{ z_{j}\right\} _{j=0}^{N-1}$ are equally distributed
on $\cir{\kappa r}$, that is \[
z_{j}=\kappa r\exp\left(\frac{2\pi ij}{N}\right).\]
Also $m$ is the normalized angular measure on $\cir r$. Under this
conditions we have
\begin{lem}
\label{lem:approx_log_int}For normal values of $r$, and outside
an exceptional set of probability at most $2\cdot\exp\left(-S(\kappa r)\right),$
we have\begin{equation}
\left|\frac{1}{N}\sum_{j=1}^{N}\log|f(z_{j})|-\intop_{\cir r}\log|f|\, dm\right|\le\frac{C}{\delta^{4}N}\log\mu(r).\label{eq:error_in_discrete_approx}\end{equation}
\end{lem}
\begin{proof}
Denote by $P_{j}(z)=P(z,z_{j})$ the Poisson kernel for the disk $\disk r,$
$|z|=r$, $|z_{j}|<r$. Since $\log|f|$ is a harmonic function we
have\begin{eqnarray*}
\frac{1}{N}\sum_{j=1}^{N}\log|f(z_{j})| & = & \intop_{\cir r}\left(\frac{1}{N}\sum_{j=1}^{N}P_{j}\right)\log|f|\, dm\\
 & = & \intop_{\cir r}\log|f|\, dm+\intop_{\cir r}\left(\frac{1}{N}\sum_{j=1}^{N}P_{j}-1\right)\log|f|\, dm.\end{eqnarray*}
The last expression can be estimated by \begin{equation}
\intop_{\cir r}\left(\frac{1}{N}\sum_{j=1}^{N}P_{j}-1\right)\log|f|\, d\mu\le\max_{z\in\cir r}\left|\frac{1}{N}\sum_{j=1}^{N}P_{j}-1\right|\cdot\intop_{\cir r}\left|\log|f|\right|\, dm.\label{eq:log_int_error_terms}\end{equation}
For the first factor in the RHS of \eqref{eq:log_int_error_terms},
we start with\[
\intop_{\cir{\kappa r}}P(z,\omega)\, dm(\omega)=1,\]
and then split the circle $\cir{\kappa r}$ into a union of $N$ disjoint
arcs $I_{j}$ of equal angular measure $\mu(I_{j})=\frac{1}{N}$ centered
at the $z_{j}$'s. Then\[
1=\frac{1}{N}\sum_{j=1}^{N}P(z,z_{j})+\sum_{j=1}^{N}\intop_{I_{j}}\left(P(z,\omega)-P(z,z_{j})\right)\, dm(\omega),\]
and\begin{eqnarray}
|P(z,\omega)-P(z,z_{j})| & \le & \max_{\omega\in I_{j}}|\omega-z_{j}|\cdot\max_{z,\omega}|\nabla_{\omega}P(z,\omega)|\nonumber \\
 & \le & \frac{2\pi r}{N}\cdot\frac{Cr}{(r-|\omega|)^{2}}\le\frac{C}{\delta^{2}N}.\label{eq:Poisson_kernel_upper_bnd}\end{eqnarray}
For the second factor on the RHS of \eqref{eq:log_int_error_terms},
using Lemma \ref{lem:Dev-M-Very-Small}, we may suppose that there
is a point $a\in\kappa\cir r$ such that $\log|f(a)|\ge0$ (discarding
an exceptional event of probability at most $\exp\left(-S(\kappa r)\right)$).
Then we have\[
0\le\intop_{\cir r}P(z,a)\log|f(z)|\, dm(z),\]
and hence\[
\intop_{\cir r}P(z,a)\log^{-}|f(z)|\, dm(z)\le\intop_{\cir r}P(z,a)\log^{+}|f(z)|\, dm(z).\]
For $|z|=r$ and $|a|=\kappa r$ we have,\[
\frac{\delta}{2}\le\frac{1-(1-\delta)}{1+(1-\delta)}\le P(z,a)\le\frac{1+(1-\delta)}{1-(1-\delta)}\le\frac{2}{\delta}.\]
By Lemma \ref{lem:Dev-M-Upp-Bnd}, outside a very small exception
set (of the order $\exp\left(-\mu(r)\right)$), we have $\log\M(r)\le2\cdot\log M(r)$,
and we notice that from \eqref{eq:S(r):upp:bnd} it follows that $\mu(r)$
is much bigger than $S(\kappa r)$, so this exceptional set is indeed
small. Therefore\[
\intop_{\cir r}\log^{+}|f|\, d\mu\le2\log M(r).\]
Now we have\[
\intop_{\cir r}\log^{-}|f|\, d\mu\le\frac{C}{\delta^{2}}\log M(r).\]
Finally (and using \eqref{eq:WV_main_thm}) \begin{equation}
\intop_{\cir r}\left|\log|f|\right|\, d\mu\le\frac{C}{\delta^{2}}\log M(r)\le\frac{C}{\delta^{2}}\log\mu(r).\label{eq:abs_log_int_upper_bnd}\end{equation}
Combining \eqref{eq:Poisson_kernel_upper_bnd} and \eqref{eq:abs_log_int_upper_bnd}
we get the result.
\end{proof}

\subsection{Deviations for the logarithmic integral}

We recall that if\[
f(z)=\sum_{n=0}^{\infty}\phi_{n}a_{n}z^{n},\]
where $\phi_{n}$ are i.i.d standard complex Gaussian random variables,
then the vector $\left(f(z_{1}),\ldots,f(z_{N})\right)$ has a multivariate
complex Gaussian distribution, with covariance matrix:\begin{equation}
\Sigma_{ij}=\Cov{f(z_{i}),f(z_{j})}=\Ex(f(z_{i})\overline{f(z_{j})})=\sum a_{k}^{2}\left(z_{i}\bar{z_{j}}\right)^{k}.\label{eq:def_cov_matrix}\end{equation}
The density function of a multivariate complex Gaussian distribution
is:\[
\zeta\mapsto\frac{1}{\pi^{N}\det\Sigma}\exp(-\zeta^{*}\Sigma^{-1}\zeta).\]
We introduce the set ($\log_{2}\mu(r)\equiv\log\log\mu(r)$)\begin{equation}
\A^{\prime}=\left\{ \zeta\in\bbc^{N}\,:\,\prod_{j=1}^{N}|\zeta_{j}|\le\exp\left(2N\log_{2}\mu(r)+C\delta^{-4}\log\mu(r)\right)\right\} \label{eq:def_set_a_prime}\end{equation}
and denote by $\B$ the set where estimate \eqref{eq:error_in_discrete_approx}
in Lemma \ref{lem:approx_log_int} holds. We abuse notation by writing\[
\pr{\A^{\prime}}=\pr{\left(f(z_{1}),\ldots,f(z_{N})\right)\in\A^{\prime}}.\]
Using this notation we get the simple
\begin{lem}
\label{lem:log_int_deviations}\[
\pr{\int_{\cir r}\log\,|f(z)|\, dm\le2\log_{2}\mu(r)}\le\pr{\A^{\prime}}+\pr{\B^{c}}.\]
\end{lem}
\begin{proof}
We start by discarding the exceptional set in Lemma \ref{lem:approx_log_int},
this adds the term $\pr{\B^{c}}$. Now we can assume that\[
\frac{1}{N}\sum_{j=1}^{N}\log|f(z_{j})|\le\intop_{\cir r}\log|f|\, dm+\frac{C}{\delta^{4}}\cdot\frac{\log\mu(r)}{N},\]
or\[
\prod_{j=1}^{N}|f(z_{j})|\le\exp\left(N\cdot\intop_{\cir r}\log|f|\, dm+\frac{C}{\delta^{4}}\log\mu(r)\right).\]
In terms of probabilities we can write\[
\pr{\int_{\cir r}\log\,|f(z)|\, dm\le2\log_{2}\mu(r)}\le\pr{\B^{c}}+\pr{\A^{\prime}}.\]

\end{proof}
Before we continue, we need two asymptotic estimates.
\begin{lem}
\label{lem:determinant_cov_matrix_low_bnd}Let $\Sigma$ be the covariance
matrix defined in \eqref{eq:def_cov_matrix}. Choose $N=N_{1}(r)$,
then we have the following estimate\[
\log\left(\det\Sigma\right)\ge S(\kappa r).\]
\end{lem}
\begin{proof}
Notice that we can represent $\Sigma$ in the following form\[
\Sigma=V\cdot V^{*}\]
where \[
V=\left(\begin{matrix}a_{0} & a_{1}\cdot z_{1} & \dots & a_{n}\cdot z_{1}^{n} & \dots\\
\vdots & \vdots & \vdots & \vdots & \dots\\
a_{0} & a_{1}\cdot z_{N} & \dots & a_{n}\cdot z_{N}^{n} & \dots\end{matrix}\right).\]
We observe that since $a(t)$ is a log-concave function, it follows
that $n\mapsto a_{n}\cdot r^{n}$ is a unimodal sequence, and therefore
$\N_{1}(r)=\left\{ 0,1,\ldots,N_{1}(r)-1\right\} $. Therefore we
can estimate the determinant of $\Sigma$ by projecting $V$ on the
first $N_{1}(r)$ coordinates (let's denote this projection by $P$).
Since $\det\Sigma$ is the square of the product of the singular values
of $V$, and these values are only reduced by the projection, we have\[
\det\Sigma\ge\left(\det PV\right)^{2}=\left|\begin{array}{cccc}
a_{0} & a_{1}z_{1} & \ldots & a_{N-1}z_{1}^{N-1}\\
\vdots & \vdots & \vdots & \vdots\\
a_{0} & a_{1}z_{N} & \ldots & a_{N-1}z_{N}^{N-1}\end{array}\right|^{2}\]
and so\begin{eqnarray*}
\det\Sigma & \ge & \prod_{n\in\N_{1}(r)}a_{n}^{2}\cdot\left|\begin{array}{cccc}
1 & z_{1} & \ldots & z_{1}^{N-1}\\
\vdots & \vdots & \vdots & \vdots\\
1 & z_{N} & \ldots & z_{N}^{N-1}\end{array}\right|^{2}\\
 & = & \prod_{n\in\N_{1}(r)}a_{n}^{2}\cdot\prod_{1\le i\ne j\le N}\left|z_{i}-z_{j}\right|\\
 & = & \Pi_{1}\cdot\Pi_{2}\end{eqnarray*}
The $z_{i}$'s are the roots of the equation $z^{N}=\left(\kappa r\right)^{N}$,
denoting $z_{1}=\kappa r$ we get\[
\prod_{i=2}^{N}(z_{1}-z_{i})=N\left(\kappa r\right)^{N-1},\]
and\[
\Pi_{2}=\prod_{1\le i\ne j\le N}|z_{i}-z_{j}|=\left(\prod_{i=2}^{N}\left|z_{1}-z_{i}\right|\right)^{N}=\left(\kappa r\right)^{N(N-1)}N^{N}.\]
We now {}``collect'' the product of the $\kappa r$'s, rewrite it
as $\left(\kappa r\right)^{N(N-1)}=\prod_{n=0}^{N-1}\left(\kappa r\right)^{2n}$
and get\[
\det\Sigma\ge\prod_{n=0}^{N-1}a_{n}^{2}\left(\kappa r\right)^{2n}=\exp\left(S\left(\kappa r\right)\right).\]

\end{proof}
We denote by $\A$ the following set (see \eqref{eq:def_set_a_prime}
for the definition of the set $\A^{\prime}$)

\begin{equation}
\A=\left\{ \zeta\in\bbc^{N}\,:\,\zeta\in\A^{\prime}\mbox{ and }|\zeta_{j}|\le M^{2}(r),\quad0\le j\le N-1\right\} \label{eq:def_set_a}\end{equation}
and by $I$ the following quantity\begin{equation}
I=\pi^{-N}\cdot\volcn{\A}.\label{eq:Asymp-Int}\end{equation}
We use the following lemma (see \cite[Lemma 11]{Nis}) to estimate
$I$:
\begin{lem}
Set $s>0$, $t>0$ and $N\in\bbn^{+}$, such that $\log\left(t^{N}/s\right)\ge N$.
Denote by $\C_{N}$ the following set\[
\C_{N}=\C_{N}\left(t,s\right)=\left\{ r=\left(r_{1},\ldots,r_{N}\right)\::\:0\le r_{j}\le t,\,\prod_{1}^{N}r_{j}\le s\right\} .\]
Then\[
\volrn N{\C_{N}}\le\frac{s}{\left(N-1\right)!}\log^{N}\left(t^{N}/s\right).\]

\end{lem}
Now we have as an almost immediate
\begin{cor}
\label{cor:prob_integral_upper_bnd}Suppose that $r$ is normal and
large enough and that $\delta$ satisfies $\delta^{-4}=o\left(N_{1}(r)\right)$,
then we have\[
\log I\le C\cdot N_{1}(r)\log_{2}\mu(r)+C\delta^{-4}\log\mu(r).\]
\end{cor}
\begin{proof}
Set $N=N_{1}(r)$ and recall that\[
\A=\left\{ \begin{array}{cc}
 & |\zeta_{j}|\le M^{2}(r),\quad0\le j\le N-1\\
\zeta\,:\, & \mbox{and}\\
 & \prod_{j=1}^{N}|\zeta_{j}|\le\exp\left(2N\log\log\mu(r)+C\delta^{-4}\log\mu(r)\right)\end{array}\right\} .\]
To shorten the expressions above, we write \[
s=\exp\left(2N\log\log\mu(r)+C\delta^{-4}\log\mu(r)\right),\quad t=M^{2}(r).\]
We want to translate the integral $I$ into an integral in $\bbr^{N}$,
using the change of variables $\zeta_{j}=r_{j}\cos(\theta_{j})+ir_{j}\sin(\theta_{j})$.
Integrating out the variables $\theta_{j}$, we get $I^{\prime}=2^{N}\intop_{\C}\prod r_{j}\, dr$,
where the new domain is\[
\C=\left\{ r=\left(r_{1},\ldots,r_{N}\right)\::\:0\le r_{j}\le t,\,\prod_{j=1}^{N}r_{j}\le s\right\} .\]
We can find an explicit expression for this integral, but, instead
we will simplify it even more to\begin{equation}
I^{\prime}\le2^{N}s\cdot\volrn N{\C}\label{eq:I_prime_upp_bnd}\end{equation}
Now, in order to use the previous lemma, we have to check the condition
$\log\left(t^{N}/s\right)\ge N$, or (where $C>0$)\[
2N_{1}(r)\log M(r)-2N_{1}(r)\log_{2}\mu(r)-C\delta^{-4}\log\mu(r)\ge N_{1}(r),\]
which is satisfied under our assumptions, for $r$ large enough. After
applying the lemma, we get (for $r$ large enough)\begin{eqnarray*}
I^{\prime} & \le & \frac{N\cdot2^{N}s^{2}}{N!}\log^{N}\left(t^{N}/s\right)\\
 & \le & \frac{s^{2}e^{2N}}{N^{N}}\log^{N}\left(t^{N}/s\right)\\
 & = & \exp\left(2\log s+N\log_{2}t+2N-N\log_{2}s\right)\\
 & \le & \exp\left(2\log s+N\log_{2}t\right).\end{eqnarray*}
Recalling the definitions of $s$ and $t$, we finally get

\begin{eqnarray*}
\log I^{\prime} & \le & 4N\log_{2}\mu(r)+C\delta^{-4}\log\mu(r)+N\log_{2}M(r)+C\\
 & \le & C_{1}N_{1}(r)\log_{2}\mu(r)+C_{2}\delta^{-4}\log\mu(r).\end{eqnarray*}

\end{proof}
We now continue to estimate probabilities of the events $\A$ and
$\A^{\prime}$ introduced in \eqref{eq:def_set_a} and \eqref{eq:def_set_a_prime}.
\begin{lem}
\label{lem:prob_estimates}With $r$ and $\delta$ satisfing the conditions
of Corollary \ref{cor:prob_integral_upper_bnd}, we have the following
estimates:\textup{\[
\pr{\A^{\prime}\backslash\A}\le\exp\left(-c\mu(r)\right)\]
and\[
\pr{\A}\le\exp\left(-S(\kappa r)+C_{1}N_{1}(r)\log_{2}\mu(r)+C_{2}\delta^{-4}\log\mu(r)\right).\]
}\end{lem}
\begin{proof}
If $\zeta\in\A^{\prime}\backslash\A$ then for some $j$ we have $\left|f(z_{j})\right|=|\zeta_{j}|>M^{2}(r)$.
Using Lemma \ref{lem:Dev-M-Upp-Bnd}, with the choice $\sigma=\frac{1}{2}$,
we see that this event has a probability at most $\exp\left(-c\mu(r)\right)$.
For the second estimate we need to bound from above the integral\[
\intop_{\A}\frac{1}{\pi^{N}\det\Sigma}\exp\left(-\zeta^{*}\Sigma^{-1}\zeta\right)\, d\zeta.\]
Discarding the exponential function and using Lemma \ref{lem:determinant_cov_matrix_low_bnd}
and Corollary \ref{cor:prob_integral_upper_bnd}, we get \[
\pr{\A}\le\frac{\volcn{\A}}{\pi^{N}\det\Sigma}\le\exp\left(-S(\kappa r)+C_{1}N_{1}(r)\log_{2}\mu(r)+C_{2}\delta^{-4}\log\mu(r)\right).\]

\end{proof}

\subsection{Lower bound for $p_{H}$}

We collect all the previous results into the proof of Proposition
\ref{thm:p_H_low_bnd}
\begin{proof}
Suppose that $f(z)$ has no zeros inside $\disk r$, then\[
\int_{\cir r}\log|f(z)|\, dm=\log|f(0)|.\]
We can use the fact that $\log|f(0)|$ cannot be too large, in fact\[
\pr{\log|f(0)|\ge2\log_{2}\mu(r)}=\pr{|\phi_{0}|\ge\log^{2}\mu(r)}\le\exp\left(-\log^{4}\mu(r)\right).\]
Now combining Lemma \ref{lem:log_int_deviations} and Lemma \ref{lem:prob_estimates},
we see that the probability of the event $\left\{ f(z)\ne0\mbox{ in }\mbox{\ensuremath{\disk r}}\right\} $
does not exceed

\begin{eqnarray*}
\exp\left(-\log^{4}\mu(r)\right) & + & \exp\left(-\mu(r)\right)\\
 & + & 2\exp\left(-S(\kappa r)\right)\\
 & + & \exp\left(-S(\kappa r)+C_{1}N_{1}(r)\log_{2}\mu(r)+C_{2}\delta^{-4}\log\mu(r)\right).\end{eqnarray*}
Since by \eqref{eq:S(r):upp:bnd} the functions $\mu(r)$ and $\log^{4}\mu(r)$
are much bigger than $S(\kappa r)$, we have the required estimate\begin{equation}
p_{H}(r)\ge S(\kappa r)-C_{1}N_{1}(r)\log_{2}\mu(r)-C_{2}\delta^{-4}\log\mu(r).\label{eq:p_H_low_bnd_w_error}\end{equation}

\end{proof}

\section{Proofs of Theorems \ref{thm:p_H_raw_asymp} and \ref{thm:p_H_exact_asymp}
\label{sec:Proof-of-Main-Theorem}}

In this section we prove Theorem \ref{thm:p_H_raw_asymp} using the
lower and upper bound estimates from the previous sections. We also
estimate the size of the error terms, for functions with sufficient
growth rate, and prove Theorem \ref{thm:p_H_exact_asymp}.

\subsection{Proof of Theorem \ref{thm:p_H_raw_asymp}}

By Proposition \ref{thm:p_H_upp_bnd} the error term for the upper
bound is (using \eqref{eq:S(r):low:bnd})\[
N_{1}(r)\log N_{1}(r)\le C\cdot N_{1}(r)\log_{2}\mu(r)=o\left(S(r)\right),\]
so it is indeed small.

For the lower bound \eqref{eq:p_H_low_bnd_w_error}, we start by selecting
$\delta$ in the following way\begin{equation}
\delta=\left(N_{1}(r)\right)^{-1/5},\label{eq:delta_selection}\end{equation}
now by Proposition \ref{thm:p_H_low_bnd} the error term is\[
C_{1}N_{1}(r)\log_{2}\mu(r)+C_{2}\delta^{-4}\log\mu(r)\]
and we see that the error term is asymptotically smaller than $S(r)$.
What is left is to show that $S(\kappa r)$ is close to $S(r)$.
\begin{lem}
\label{lem:S(r')_approx}Set $r^{\prime}=(1-\delta)r$, for normal
values of $r$ which are large enough, \[
S(r^{\prime})\ge S(r)-C\left(N_{1}(r)\right)^{9/5}\]
and\[
\left(N_{1}(r)\right)^{9/5}=o\left(S(r)\right)\]
\end{lem}
\begin{proof}
We notice that for $\delta<\frac{1}{2}$ we have\[
\log\left(1-\delta\right)\ge-\delta-\delta^{2}.\]
Then it follows that (notice that $\N_{1}(r^{\prime})\subset\N_{1}(r)$)\begin{eqnarray*}
\frac{S(r)-S(r^{\prime})}{2} & = & \sum_{n\in\N_{1}(r)\backslash\N_{1}(r^{\prime})}\log a_{n}r^{n}+\sum_{n\in\N_{1}(r^{\prime})}\left(\log a_{n}r^{n}-\log a_{n}\left(r^{\prime}\right)^{n}\right)\\
 & \le & \Sigma_{1}+\Sigma_{2}.\end{eqnarray*}
For the first sum we notice that if $n\in\N_{1}(r)\backslash\N_{1}(r^{\prime})$
then\begin{eqnarray*}
0 & \ge & \log a_{n}\left(r^{\prime}\right)^{n}\ge\log a_{n}r^{n}-n\left(\delta+\delta^{2}\right)\\
 & \Downarrow\\
\log a_{n}r^{n} & \le & n\left(\delta+\delta^{2}\right)\le2N_{1}(r)\delta\le2\left(N_{1}(r)\right)^{4/5}\end{eqnarray*}
and so\[
\Sigma_{1}\le2\left(N_{1}(r)-N_{1}(r^{\prime})\right)\left(N_{1}(r)\right)^{4/5}\le2\left(N_{1}(r)\right)^{9/5}.\]
For the second sum we have\[
\Sigma_{2}\le\left(N_{1}(r)\right)^{2}\left(-\log\left(1-\delta\right)\right)\le2\left(N_{1}(r)\right)^{9/5},\]
and overall we get the required estimate.

Now we will prove that $N_{1}(r)\le C\sqrt{S(r)}\log S(r)$, which
will give us the second claim. We start with\[
S(r)\stackrel{\eqref{eq:S(r):upp:bnd}}{\le}C\log^{2}\mu(r)\log_{2}^{2}\mu(r)\le C\log^{2}\mu(r)\log^{2}S(r),\]
therefore\[
\sqrt{S(r)}\le C\log\mu(r)\log S(r)\]
or\begin{equation}
N_{1}(r)\stackrel{\eqref{eq:S(r):low:bnd}}{\le}2\cdot\frac{S(r)}{\log\mu(r)}\le C\sqrt{S(r)}\log S(r)\label{eq:N_1_upp_bnd_in_S_r}\end{equation}
and so\[
\left(N_{1}(r)\right)^{9/5}\le C\left(S(r)\right)^{9/10}\left(\log S(r)\right)^{9/5}=o\left(S(r)\right).\]
This concludes the proof of Lemma \ref{lem:S(r')_approx} and Theorem
\ref{thm:p_H_raw_asymp}.
\end{proof}

\subsection{Proof of Theorem \ref{thm:p_H_exact_asymp}}

We need the following
\begin{lem}
Let $\gamma>0$. Suppose that $\log\mu(r)\ge\exp\left(\log^{\gamma}r\right)$,
then for values of $r$ which are normal and large enough \[
\left(N_{1}(r)\right)^{4/5}\log\mu(r)\le C\left(S(r)\right)^{9/10}\log^{1/\gamma+8/5}S(r).\]
\end{lem}
\begin{proof}
We first notice that from the assumption on $\log\mu(r)$, we have
(for $r$ large enough)\[
\log_{2}^{8/5}\mu(r)\le\frac{\log_{2}^{1/\gamma+8/5}\mu(r)}{\log^{9/10}r}.\]
Now (by Lemma \ref{cor:N_1-upper_bnd}, Lemma \ref{lem:N_1_low_bound}
and \ref{eq:S(r):low:bnd})\begin{eqnarray*}
\left(N_{1}(r)\right)^{4/5}\log\mu(r) & \le & C\log^{9/5}\mu(r)\log_{2}^{8/5}\mu(r)\\
 & \le & C\left(\frac{\log^{2}\mu(r)}{\log r}\right)^{9/10}\log_{2}^{1/\gamma+8/5}\mu(r)\\
 & \le & C\left(\frac{1}{2}\cdot N_{1}(r)\log\mu(r)\right)^{9/10}\log_{2}^{1/\gamma+8/5}\mu(r)\\
 & \le & C\left(S(r)\right)^{9/10}\log^{1/\gamma+8/5}S(r).\end{eqnarray*}

\end{proof}
Let $\alpha\ge1$. To finish the proof of Theorem \ref{thm:p_H_exact_asymp}
we note that if $a(t)\ge\exp\left(-t\log^{\alpha}t\right)$ then for
$r$ large enough, we have\[
\log\mu(r)\ge\max_{n\in\bbn}\left[-n\log^{\alpha}n+n\log r\right]\ge c_{1}\exp\left(c_{2}\left(\log r\right)^{1/\alpha}\right),\]
for example by selecting $n$ in such a way that it will satisfy $\log^{\alpha}n\approxeq\frac{1}{2}\log r$.
Finally, we see that $\log\mu(r)$ satisfies the condition in the
previous lemma.$\hfill\Box$\medskip{}

We notice that using our methods, Theorem \ref{thm:p_H_exact_asymp}
cannot be proved for arbitrary (log-concave) coefficients. The problem
comes from the following error term in the lower bound\[
\delta^{-4}\log\mu(r)=\left(N_{1}(r)\right)^{4/5}\log\mu(r).\]
To see that we cannot bound it by an expression of the form $\left(S(r)\right)^{\alpha}$
with $\alpha<1$, we take $a(t)=\exp\left(-\exp\left(t\right)\right)$.
For this function we have\begin{eqnarray*}
N_{1}(r) & = & \log_{2}r+\log_{3}r+\O\left(1\right),\\
\log\mu(r) & = & \log r\log_{2}r+\O\left(\log r\right),\\
S(r) & = & \Theta\left(\log r\log_{2}^{2}r\right).\end{eqnarray*}
We see that for every $\epsilon>0$ (for $r$ large enough) \[
\frac{S(r)}{\left(N_{1}(r)\right)^{4/5}\log\mu(r)}=\Theta\left(\log_{2}^{1/5}r\right)=o\left(\left(S(r)\right)^{\epsilon}\right).\]

\end{document}